\documentclass[12pt]{amsart}
\usepackage{amsfonts}
\usepackage{graphicx}
\usepackage{amscd}
\usepackage{amsmath}
\usepackage{amssymb}
\usepackage{amsfonts}
\usepackage[margin=1.2in]{geometry}

\newtheorem{theorem}{Theorem}[section]
\theoremstyle{plain}
\newtheorem{corollary}[theorem]{Corollary}
\newtheorem{lemma}[theorem]{Lemma}
\newtheorem{proposition}[theorem]{Proposition}

\theoremstyle{definition}
\newtheorem{definition}[theorem]{Definition}

\newtheorem{remark}[theorem]{Remark}

\numberwithin{equation}{section}

\begin{document}

\title{Besov regularity in non-linear generalized functions}

 \author[S. Pilipovi\'{c}]{Stevan Pilipovi\'c}
 \thanks{The research of S. Pilipovi\'{c} was supported by the Project F10 (SP) 
of the Serbian Academy of Sciences and Arts}
 \address{Department of Mathematics and Informatics\\ University of Novi Sad\\ Trg Dositeja Obradovi\' ca 4\\ 21000 Novi Sad, Serbia}
\email {stevan.pilipovic@dmi.uns.ac.rs}

  \author[D. Scarpal\'{e}zos]{Dimitris Scarpal\'{e}zos}
\address{Centre de Math\' ematiques de Jussieu\\
Universit\' e Paris 7 Denis Diderot\\ Case Postale 7012, 2, place
Jussieu\\ F-75251 Paris Cedex 05, France}
    \email{dim.scarpa@gmail.com}   

 \author[J. Vindas]{Jasson Vindas}
\thanks{J. Vindas gratefully acknowledges support by Ghent University through the BOF-grant 01J04017 and by the Research Foundation--Flanders through the FWO-grant G067621N}
\address{Department of Mathematics: Analysis, Logic and Discrete Mathematics\\ Ghent University\\ Krijgslaan 281\\ 9000 Gent\\ Belgium}
\email{jasson.vindas@UGent.be}


 \begin{abstract}
We introduce and study new modules and spaces of generalized
functions that are related to the classical Besov spaces. 
 Various Schwartz distribution spaces are naturally embedded
 into our new generalized function spaces.
 We obtain precise criteria for detecting Besov regularity of distributions. 
\end{abstract}

\maketitle

\section{Introduction}\label{sec0}
Algebras of generalized functions supply effective tools for studying singular problems in analysis, geometry, and mathematical physics \cite{bia,col1,gkos,nps,ober001}.  Since the Colombeau type algebras contain the spaces of Schwartz distributions, one of the most standard and central questions in the area is then to find out whether a given generalized function is actually a classical ``smooth function". For example, this natural question arises when solving singular PDE.
The regularity theory for non-linear generalized functions was initiated by Oberguggenberger  \cite{ober001} with the introduction of his algebra of regular generalized functions, which provides the foundation for microlocal analysis in the Colombeau setting \cite{gah,kuhe,nps}; see also \cite{dvv18} for microlocal analysis with respect to Denjoy-Carlemann classes. In his pioneer work H\"{o}rmann \cite{hoermann} provided criteria for detecting H\"{o}rderian regularity; his study has been completed in \cite{piscavi,psv} with precise characterizations of H\"{o}lder-Zygmund classes inside the Colombeau algebra. 

The goal of this paper is to provide a non-linear framework for regularity theory based on the well-known Besov spaces \cite{peetre76}. It turns out that the classical (special) Colombeau algebra is not the right setting to work with Besov regularity. In Section \ref{section q-generalized}, we introduce and study a new scale of spaces of generalized functions, the spaces $\mathcal{G}_{q}(\Omega)$ where $\Omega$ is an open subset of $\mathbb{R}^{d}$ and the parameter $q\in[1,\infty]$ will be related to the Besov spaces. They are constructed in the same spirit as the Colombeau algebra, namely, as quotients of spaces of ``moderate" and ``negligible'' nets of smooth functions admitting embeddings of the space of distributions and respecting the multiplication of $C^{\infty}$-functions. However, unlike for Colombeau generalized functions, moderateness and negligibility are not given by pointwise bounds if $q<\infty$, but rather defined in terms of $L^q$-integrability bounds with respect to the net parameter.  Only when $q=\infty$ we obtain an algebra, and the construction in fact coincides with the so-called algebra of generalized functions with smooth parameter dependence $\mathcal{G}_{co}(\Omega)$ introduced by Burtscher and Kunzinger in \cite{bucu}. In all other cases, our spaces are not algebras, but nevertheless this is compensated by the fact that they are differential modules over  the differentiable algebra $\mathcal{G}_{co}(\Omega)$.

In Section \ref{section l besov gf} we introduce a family of subspaces of $\mathcal{G}_{q}(\Omega)$ that will allow us to characterize local membership to a given Besov space in a precise fashion. Our main regularity result is Theorem \ref{1mtheorem}. We present several other useful regularity criteria in Section \ref{section association}, where the hypotheses involve a generalization of the concept of strong association \cite{ps}. The article concludes with global counterparts of our results; in particular, we obtain in Section \ref{section global} a global characterization of the Besov spaces. 

\section{Preliminaries}
\label{preli Besov}
\subsection{Notation} \label{notation}Throughout the article $\Omega$ stands for an open subset of $\mathbb{R}^{d}$ and we always assume  that the parameters $p,q\in[1,\infty]$.  We consider the family of Sobolev norms
$||\phi||_{W^{k,p}(\Omega)}= \max
\{||\phi^{(\alpha)}||_{L^{p}(\Omega)}:\:|\alpha|\leq m \},$ where $k\in
{\mathbb N}_0$.  We indistinctly denote the space of smooth functions on $\Omega$ by 
$C^{\infty}(\Omega)=\mathcal{E}(\Omega).$ The notation $\omega\Subset \Omega$ means that $\omega$ is open and has compact closure inside $\Omega$. If $X$ is a locally convex space of distributions that is closed under multiplication by compactly supported smooth functions, we define its associated local space on $\Omega$ as $X_{loc}(\Omega)=\{T\in\mathcal{D}'(\Omega) : \chi \cdot T\in X \mbox{ for all }\chi\in\mathcal{D}(\Omega)\}$.
Finally, we also fix a mollifier $\phi\in\mathcal{S}(\mathbb{R}^{d})$ such that
\begin{equation}
\label{eq mollifier}
\int_{\mathbb{R}^{d}}\phi(x)dx=1 \qquad\mbox{and} \qquad \int_{\mathbb{R}^{d}}x^{\alpha}\phi(x)dx=0, \quad \mbox{for all }\alpha\in\mathbb{N}_{0}^{d}\setminus\{0\},
\end{equation}
and consider its associated delta net $(\phi_{\varepsilon})_{\varepsilon\in(0,1)}$, where $\phi_{\varepsilon}=\varepsilon^{-d}\phi(\cdot/\varepsilon)$.

\subsection{Algebras of generalized functions}
 The standard (special) Colombeau algebra of generalized functions is defined as the quotient $\mathcal{G}(\Omega)=\mathcal{E}_{M}(\Omega)/\mathcal{N}(\Omega),$ where 
the algebra (of moderate nets) $\mathcal{E}_{M}(\Omega)$  and its ideal (of negligible nets) $\mathcal  N(\Omega)$
respectively consist of nets $(f_{\varepsilon
})_{\varepsilon\in(0,1)}=(f_{\varepsilon
})_{\varepsilon}\in \mathcal{E}(\Omega)^{(0,1)}$ with the
properties
\begin{equation}\label{2drs}
(\forall k\in\mathbb{N}_0)(\forall
\omega\subset\subset\Omega)(\exists a\in\mathbb{R})
(||f_{\varepsilon}||_{W^{k,p}(\omega)}=O(\varepsilon^{a})),
\end{equation}
\[
\mbox{ resp., } \;(\forall k\in\mathbb{N}_0)(\forall
\omega\subset\subset \Omega)(\forall b\in\mathbb{R})
(||f_{\varepsilon}||_{W^{k,p}(\omega)}=O(\varepsilon^{b}))
\]
(big $O$ and small $o$ are the Landau symbols). The well-known Sobolev embedding theorems guarantee that these definitions are independent of $p$, being $p=\infty$ the standard choice in the literature \cite{gkos}. The ring of generalized constants $\widetilde{\mathbb{C}}$ is obtained by considering the subrings of nets of $\mathcal{E}_{M}(\Omega)$  and $\mathcal  N(\Omega)$
consisting of only constant functions and then forming the corresponding quotient. If one further restricts to nets of real valued constant functions, the resulting quotient algebra is $\widetilde{\mathbb{R}}$. The space of generalized functions $\mathcal{G}(\Omega)$  then becomes a module over the ring $\widetilde{\mathbb{C}}$. We refer to \cite{bia,col1,gkos} for more details on Colombeau algebras. See also \cite{dvv18,dvv19} for non-linear theories of ultradistributions and hyperfunctions.

The representatives $(f_{\varepsilon})_{\varepsilon}$ of an equivalence class $[(f_{\varepsilon})_{\varepsilon}]\in \mathcal{G}(\Omega)$ are allowed to depend arbitrarily on the parameter $\varepsilon$. Following \cite{bucu}, it is also meaningful to consider algebras of generalized functions with continuous or smooth dependence on the parameter $\varepsilon$. In fact, if in addition to \eqref{2drs} one imposes in the definitions of moderate and negligible nets that the functions $(\varepsilon,x)\mapsto f_{\varepsilon}(x)$ should be\footnote{Equivalently, the vector-valued mapping $\varepsilon\mapsto f_{\varepsilon}$ belongs to $C((0,1);\mathcal{E}(\Omega))$ or $C^{\infty}((0,1);\mathcal{E}(\Omega))$, respectively.} continuous in $\varepsilon$ and smooth in $x$, or respectively smooth in both variables $(\varepsilon,x)$, one denotes the corresponding so resulting quotient algebras as $\mathcal{G}_{co}(\Omega)$ and $\mathcal{G}_{sm}(\Omega)$. We then obtain canonical  algebra monomorphisms $\mathcal{G}_{sm}(\Omega) \to \mathcal{G}_{co}(\Omega)\to\mathcal{G}(\Omega)$. We might simply then identify the algebra in the left-hand side of the arrow as a subalgebra of the one in the right-hand side. Interestingly, it has been shown in \cite{bucu} that $\mathcal{G}_{sm}(\Omega)=\mathcal{G}_{co}(\Omega)\subsetneq \mathcal{G}(\Omega)$. Similarly, one can also consider the rings of generalized constants with continuous or smooth dependence and one has $\widetilde{\mathbb{C}}_{sm}=\widetilde{\mathbb{C}}_{co}\subsetneq \widetilde{\mathbb{C}}$ and $\widetilde{\mathbb{R}}_{sm}=\widetilde{\mathbb{R}}_{co}\subsetneq \widetilde{\mathbb{R}}$. 

\subsection{Besov spaces} There are different equivalent ways to introduce the Besov spaces. Here we follow the approach presented in \cite[Example 6.2, p.~294]{pv2019} in terms of generalized Littlewood-Paley pairs, which suits well the purposes of this article. Similar dyadic versions can be found in \cite{peetre76}.

Let $s\in\mathbb{R}$. We say that $(\varphi,\psi)$ is an \emph{LP-pair} of order $s$ if $\varphi,\psi\in\mathcal{S}(\mathbb{R}^{d})$ and if they satisfy the following compatibility conditions:
\begin{equation}
\label{lpcond1}
(\exists \sigma>0,\exists\eta\in(0,1))(|\hat{\varphi}(\xi)|>0  \mbox{ for }  \left|\xi\right|\leq \sigma \mbox{ and }|\hat {\psi}(\xi)|>0  \mbox{ for } \eta\sigma\leq\left|\xi\right|\leq \sigma)
\end{equation}
and 
\begin{equation}
\label{lpcond2}
 \int_{\mathbb{R}^{d}}t^{\alpha}\psi(t)dt=0  \: \mbox{ for } \: \left|\alpha\right|\leq\lfloor s\rfloor.
\end{equation} 
Note that when $s<0$ the condition on the moments is empty, thus the vanishing requirement \eqref{lpcond2} is dropped in this case. (For example, $(\phi,\phi)$ with $\phi$ the mollifier we have fixed in Subsection \ref{notation} is an LP-pair of order $s<0$.) Then, the Besov space $B^{s}_{p,q}(\mathbb{R}^{d})$ is defined as the Banach space of all distributions $T\in\mathcal{S}'(\mathbb{R}^{d})$ satisfying
\begin{equation}
\
\label{zeq} \left\|T\right\|_{B^{s}_{p,q}(\mathbb{R}^{d})}:=||T\ast\varphi||_{L^{p}(\mathbb{R}^{d})}
+\left(\int_{0}^{1}y^{-qs}||T\ast\psi_y||^{q}_{L^{p}(\mathbb{R}^{d})}\frac{dy}{y}\right)^{\frac{1}{q}}<\infty.
\end{equation}
where as usual $\psi_{y}=y^{-d}\psi(\cdot/y)$.
The definition and the norm (\ref{zeq}) (up to equivalence) are independent of the choice of the pair $(\varphi,\psi)$ 
as long as (\ref{lpcond1}) and (\ref{lpcond2}) hold \cite[Example 6.2, p.~294]{pv2019}. The flexibility to work with different LP-pairs will play an essential role in some of our arguments below. Let us remark that when $T\in\mathcal{E}'(\Omega)$, then $T\ast\varphi\in \mathcal{S}(\mathbb{R}^{d})\subset L^{p}(\mathbb{R}^{d})$ so that $T\in B^{s}_{p,q}(\mathbb{R}^{d})$ if and only if $\int_{0}^{1}y^{-qs-1}||T\ast\psi_y||^{q}_{L^{p}(\mathbb{R}^{d})}dy<\infty.$
 
\section{The new modules of generalized functions $\mathcal{G}_{q}(\Omega)$}
\label{section q-generalized}

We shall introduce the new space of generalized functions $\mathcal{G}_{q}(\Omega)$. Recall that our convention is to always assume that $p,q\in[1,\infty]$.

 We  say that a net 
 $(f_\varepsilon)\in{\mathcal E}(\Omega)^{(0,1)}$ belongs to $\mathcal E_{q}(\Omega)$, 
 respectively to $\mathcal N_{q}(\Omega)$, if $(\varepsilon,x)\mapsto f_{\varepsilon}(x)$ is also continuous with respect to $\varepsilon$  
 and  satisfies the integral growth estimates (with the obvious modification when $q=\infty$)
\begin{equation}
\label{eql1}
(\forall k\in\mathbb N_0)(\forall \omega\Subset \Omega)(\exists s\in \mathbb R)
\Big(\int_0^1\varepsilon^{qs}||f_\varepsilon||_{W^{k,p}(\omega)}^q\frac{d\varepsilon}{\varepsilon}<\infty\Big),
\end{equation}
respectively,
\begin{equation}
\label{eql2}
(\forall k\in\mathbb N_0)(\forall \omega\Subset \Omega)(\forall s\in \mathbb R)
\Big(\int_0^1\varepsilon^{qs}||f_\varepsilon||_{W^{k,p}(\omega)}^q\frac{d\varepsilon}{\varepsilon}<\infty\Big).
\end{equation}
By Sobolev's lemma, the definitions of these two spaces of nets are independent of the value of $p$. Also, one has a null description of $\mathcal{N}_{q}(\Omega)$.

\begin{lemma}
\label{lemma null ideal q} We have that $(f_{\varepsilon})_{\varepsilon}\in \mathcal{E}_{q}(\Omega)$ belongs to $\mathcal{N}_{q}(\Omega)$ if and only if
\[
(\forall \omega\Subset \Omega)(\forall s\in \mathbb R)
\Big(\int_0^1\varepsilon^{qs}||f_\varepsilon||_{L^{p}(\omega)}^q\frac{d\varepsilon}{\varepsilon}<\infty\Big).
\]
\end{lemma}
\begin{proof} Since our notions are local, we may assume that all $f_{\varepsilon}$ have support on a fixed compact $K\subset \Omega$. The result is then an immediate consequence of the multivariate version of the Landau-Kolmogorov inequality (see e.g. \cite[Eq. (4.2)]{dvv18}) and H\"{o}lder's inequality.
\end{proof}

We call the elements of $\mathcal{E}_{q}(\Omega)$ nets of smooth functions with $L^{q}$-moderate growth, while the 
ones of $\mathcal{N}_{q}(\Omega)$ will be referred as $L^{q}$-negligible nets. 
 We then set
$$\mathcal{G}_{q}(\Omega):=
\mathcal{E}_{q}(\Omega)/\mathcal{N}_{q}(\Omega).
$$

If we consider the subspaces of $\mathcal{E}_{q}(\Omega)$ and $\mathcal{N}_{q}(\Omega)$ consisting of nets of constants, respectively real constants, and then their quotient spaces, we obtain the spaces of generalized constants $\widetilde{\mathbb{C}}_{q}$ and $\widetilde{\mathbb{R}}_{q}$. We note that $\mathcal{G}_{\infty}(\Omega)=\mathcal{G}_{co}(\Omega)$, $\widetilde{\mathbb{C}}_{\infty}=\widetilde{\mathbb{C}}_{co}$, $\widetilde{\mathbb{R}}_{\infty}=\widetilde{\mathbb{R}}_{co}$, and they are algebras over the complex numbers. Furthermore, $\mathcal{G}_{q}(\Omega)$ and $\widetilde{\mathbb{C}}_{q}$ are both modules over the ring $\widetilde{\mathbb{C}}_{\infty}$. Note also that $\mathcal{G}_{q}(\Omega)$ is a module over $\mathcal{G}_{\infty}(\Omega)$. All multiplication operations as well as the action of partial derivatives are defined here in the natural way: via multiplication or differentiation of representatives. 

We now discuss the embedding of distributions into $\mathcal{G}_{q}(\Omega)$. The procedure is similar to that for classical Colombeau algebras, so we briefly sketch it.

 We first notice that using $C^{\infty}$ partitions of the unity one
can show that the functor $\omega\mapsto\mathcal{G}_q(\omega)$ is a \emph{fine sheaf} of differential modules (over the fine sheaf of differential algebras $\mathcal{G}_{\infty}$) on $\Omega$. In
fact, this is not so hard to directly verify, but we mention the proof can  considerably be
simplified by reasoning exactly as in \cite[Theorem 1.2.4]{gkos} with the aid of Lemma \ref{lemma null ideal q}.

Let $\mathcal{G}_{q,c}(\Omega)$ be the subspace of $\mathcal{G}_{q}(\Omega)$ consisting of compactly supported elements (or equivalently, those generalized functions that admit compactly supported representatives). The embedding of the compactly supported distributions is realized through
\begin{equation}
\label{embedding eq}
\iota: \mathcal{E}^{\prime}(\Omega)\to \mathcal{G}_{q,c}(\Omega), \qquad \mbox{where }\iota(T)=
[((T\ast\phi_{\varepsilon})_{|{\Omega}})_\varepsilon]\in\mathcal{G}_{q,c}(\Omega).
\end{equation}
We recall that here $\phi$ is the mollifier that we have already fixed in Subsection \ref{notation}, i.e., one that satisfies \eqref{eq mollifier}. As in the case of the classical Colombeau algebra, one has that $\iota$ is a support preserving map (and in particular injective).
\begin{lemma}
\label{lemma supp} The map \eqref{embedding eq} is support preserving.
\end{lemma}
\begin{proof} Let $T\in\mathcal{E}'(\Omega)$. The case $q=\infty$ is well-known (cf. \cite{gkos}). This already yields $\operatorname*{sup} \iota(T)\subseteq\operatorname*{sup} T$, where the first support set is now with respect to the sheaf $\mathcal{G}_{q}$. To show the reverse inclusion, assume that $\iota(T)_{|\omega}=0$ in $\mathcal{G}_{q}(\omega)$ and let $\rho\in\mathcal{D}(\omega)$. It is also well-known (\cite{gkos}) that $\langle T-T\ast \phi_{\varepsilon}, \rho\rangle=O(\varepsilon^{b})$ for each $b>0$ (which easily follows from Taylor's formula and the moment vanishing conditions from \eqref{eq mollifier}). Since $\iota(T)_{|\omega}=0$ in $\mathcal{G}_{q}(\omega)$, obtain $\langle T\ast \phi_{\varepsilon}, \rho\rangle=O(\varepsilon^{b})$ also for every $b>0$. In particular, if $0<\eta<1/2$, we have, for some $C$ not depending on $\eta$,
$$ (\log 2)| \langle T, \rho\rangle|\leq  \int_{\eta}^{2\eta} (|\langle T- T\ast \phi_{\varepsilon}, \rho\rangle|+|\langle T\ast \phi_{\varepsilon}, \rho\rangle|)\frac{d\varepsilon}{\varepsilon}\leq C\int_{\eta}^{2\eta}d\varepsilon= C\eta.
$$
Taking $\eta\to0^{+}$, we obtain $\langle T, \rho\rangle=0$ and, since $\rho$ was arbitrary, the distribution $T$ vanishes on $\omega$. This gives the other inclusion and therefore $\operatorname*{sup} \iota(T)=\operatorname*{sup} T$.
\end{proof}
 It is
a standard exercise in sheaf theory to show that any support preserving linear embedding between the
compactly supported global sections of two fine sheaves on a Hausdorff second countable
locally compact topological space can be uniquely extended to a sheaf embedding\footnote{In
fact, the latter conclusion remains valid if one replaces fineness of the sheaves by the
weaker hypothesis of softness  \cite[Lemma 2.3, p. 228]{komatsu73}.}. In particular, we obtain that the embeddings \eqref{embedding eq} give rise to a unique sheaf embedding $\iota:{\mathcal{D}'}\to \mathcal{G}_{q}$. As in the classical case, the vanishing moment conditions \eqref{eq mollifier} ensure that the restriction of $\iota$ to $C^{\infty}$ is simply given by constant nets $\iota(f)=[(f)_{\varepsilon}]$; in particular, $\iota(f\cdot g)=\iota(f)\cdot \iota(g)$ for any $f,g\in C^{\infty}(\Omega)$. Let us collect some of the properties of $\mathcal{G}_{q}$ that we have discussed so far.
\begin{proposition}\label{11nov} $\mathcal{G}_{q}$ is a fine sheaf of $\widetilde{\mathbb{C}}_{\infty}$-modules on $\mathbb{R}^{d}$. Furthermore, under multiplication induced by pointwise multiplication of representatives, $\mathcal{G}_{q}$ is a sheaf of differential modules over the sheaf of differential algebras $\mathcal{G}_{\infty}$, and in particular over $C^{\infty}$. Furthermore, we have a linear embedding $\iota:\mathcal{D}'\to \mathcal{G}_{q}$ that takes the form \eqref{embedding eq} on compact sections. Moreover, $\iota(C^{\infty}\cdot C^{\infty})= \iota(C^{\infty})\cdot \iota(C^{\infty}).$
\end{proposition}

We end this section with some remarks.

\begin{remark}
\label{rm 1 q-generalized}
Unless $q=\infty$, the modules $\widetilde{\mathbb{C}}_{q}$ and $\mathcal{G}_{q}(\Omega)$ are not algebras. In fact, consider the net of constant functions $(f_{\varepsilon})_{\varepsilon}$ given as follows. For $\varepsilon\in [n^{-1}-e^{-n},n^{-1}+e^{-n}]$ and $n\geq4$ we define $f_{\varepsilon}=n^{-2}e^{n/q}$ if $\varepsilon\in [n^{-1}-e^{-n}/2,n^{-1}+e^{-n}/2]$ and then extend it linearly on each $ [n^{-1}-e^{-n},n^{-1}-e^{-n}/2]$ and $ [n^{-1}+e^{-n}/2,n^{-1}+e^{-n}]$ in such a way that $f_{n^{-1}-e^{-n}}=f_{n^{-1}+e^{-n}}=0$.  We then set $f_{\varepsilon}=0$ elsewhere. Then, $(f_{\varepsilon})_{\varepsilon}\in \mathcal{E}_{q}(\Omega)$ but $(f^{2}_{\varepsilon})_{\varepsilon}\notin \mathcal{E}_{q}(\Omega).
$
\end{remark}

\begin{remark}
\label{rm 2 q-generalized}
The example given in Remark \ref{rm 1 q-generalized} shows that $\mathcal{E}_{q}(\Omega)\not\subset \mathcal{E}_{q'}(\Omega)$ if $q\leq q'$. On the other hand, as H\"{o}lder's inequality shows, we always have the inclusions $\mathcal{E}_{q'}(\Omega)\subset \mathcal{E}_{q}(\Omega)$ and $\mathcal{N}_{q'}(\Omega)\subset \mathcal{N}_{q}(\Omega)$ whenever $ q\leq q'$. In this case, we thus obtain a well defined canonical module homomorphism $\mathcal{G}_{q'}(\Omega)\to \mathcal{G}_{q}(\Omega)$, which is obviously not surjective. It is \emph{not injective} either, which says that, under the canonical map,  $\mathcal{G}_{q'}$ cannot be seen as a subsheaf of $\mathcal{G}_{q}(\Omega)$. To show that $\mathcal{G}_{q'}(\Omega)\to \mathcal{G}_{q}(\Omega)$ is not injective when $q'>q$ either, one needs to find a net $(g_{\varepsilon})_{\varepsilon}\in \mathcal{N}_{q}(\Omega)$ that does not belong to $(g_{\varepsilon})_{\varepsilon}\in \mathcal{N}_{q'}(\Omega)$. An example of such a net is $(g_{\varepsilon})_{\varepsilon}$ defined as the net of constant functions: For $\varepsilon\in [n^{-1}-e^{-n},n^{-1}+e^{-n}]$ and $n\geq4$ we define $g_{\varepsilon}=e^{n/q-\sqrt{n}}$ if $\varepsilon\in [n^{-1}-e^{-n}/2,n^{-1}+e^{-n}/2]$ and extend it linearly on each $ [n^{-1}-e^{-n},n^{-1}-e^{-n}/2]$ and $ [n^{-1}+e^{-n}/2,n^{-1}+e^{-n}]$ in such a way that $g_{n^{-1}-e^{-n}}=g_{n^{-1}+e^{-n}}=0$.  We then set $g_{\varepsilon}=0$ elsewhere.
\end{remark}

\begin{remark}
\label{rm 3 q-generalized}
It would also be natural to relax the condition of continuity in the parameter $\varepsilon$ in the construction of our algebras $\mathcal{G}_{q}(\Omega)$ to simply measurability. Analogously, one might strengthen the condition by asking smoothness in $\varepsilon.$
However, when $q\in[1,\infty)$, one would not gain anything new as the resulting spaces are both isomorphic to our $\mathcal{G}_{q}$ defined above. That is the reason why we have decided to impose the continuity dependence in $\varepsilon$ in this paper. To be more precise, 
let $\mathcal{E}_{q,me}(\Omega)$ and $\mathcal{N}_{q,me}(\Omega)$ be the spaces of nets $(f_{\varepsilon})_{\varepsilon}\in{\mathcal E}(\Omega)^{(0,1)}$ such that the vector-valued mapping 
\begin{equation}
\label{eql3}
(0,1)\to  \mathcal{E}(\Omega):\ \varepsilon\mapsto f_{\varepsilon} \quad \mbox{is measurable}
\end{equation}
 and \eqref{eql1} or \eqref{eql2} respectively holds. Consider then $\mathcal{G}_{q,me}(\Omega)=
\mathcal{E}_{q,me}(\Omega)/\mathcal{N}_{q,me}(\Omega).$ 
Then, the natural mapping $ \mathcal{G}_{q}(\Omega)\mapsto [(f_{\varepsilon})_{\varepsilon}]\in \mathcal{G}_{q,me}(\Omega) $ is a $\widetilde{\mathbb{C}}_{\infty}-$module monomorphism. If $q\in[1,\infty)$, it is actually an isomorphism, as follows from the next lemma. In fact, Lemma \ref{rel-s-m} shows that if one (apparently) strengthens the regularity requirement to smoothness dependence in the parameter $\varepsilon$, one actually gets nothing new either.
\end{remark}

\begin{lemma}\label{rel-s-m}
Let $q\in [1,\infty)$. Let the net $(f_\varepsilon)_\varepsilon\in\mathcal{E}(\Omega)^{(0,1)}$ satisfy \eqref{eql1} and \eqref{eql3}. Then, there is a net $(g_{\varepsilon})_{\varepsilon}$ such that the function $(\varepsilon,x)\mapsto g_\varepsilon(x)$ belongs to $C^\infty( (0,1)\times\Omega)$ and such that
\[
(\forall k\in\mathbb N_0)(\forall \omega\Subset \Omega)(\forall s>0)
\Big(\int_0^1\varepsilon^{-s}||f_\varepsilon-g_{\varepsilon}||_{W^{k,p}(\omega)}^q d\varepsilon<\infty\Big).
\]
\end{lemma}\begin{proof} We might assume $1<p<\infty$.
Let $(\omega_k)_k$ be an exhaustion of $\Omega$ by compacts, i.e., $\omega_k\Subset \omega_{k+1}$ and
$\bigcup_{k=1}^\infty\omega_k=\Omega$. 
where each $\partial\omega_k$ is smooth\footnote{It actually suffices for our purposes to assume that $\omega_k$ satisfies the segment condition \cite[p.~68]{adams}.}.
By 
(\ref{eql1}), for each $k\in\mathbb N$, there exists $N_k\in\mathbb N$  such that
\[\int_0^{1} \varepsilon^{N_k} \Vert f_\varepsilon\Vert_{W^{k,p}(\omega_{k})}^q d\varepsilon <\infty.\]
So, $\varepsilon\mapsto f_{\varepsilon}$ belongs to the vector-valued space $L^{q}((1/(k+1),1/k),\varepsilon^{N_k} d\varepsilon; W^{k,p}(\omega_k))$. Since \cite[Theorem 3.2, p.~68]{adams} (the restrictions of elements of) $\mathcal{D}(\Omega)$ are dense in $W^{k,p}(\omega_k)$, the space $\mathcal{D}((1/(k+1),1/k), \mathcal{D}(\Omega))=\mathcal{D}((1/(k+1),1/k)\times \Omega)$ is dense in $L^{q}((1/(k+1),1/k),\varepsilon^{N_k} d\varepsilon; W^{k,p}(\omega_k))$. We can thus find $(g_{\varepsilon,k})_{\varepsilon}$ with $\varepsilon\mapsto g_{\varepsilon,k}$ belonging to $\mathcal{D}((1/(k+1),1/k), \mathcal{D}(\Omega))$ such that

\[\int_{\frac1{k+1}}^{\frac1{k}}\varepsilon^{N_k}\Vert f_\varepsilon-g_{\varepsilon,k}\Vert_{W^{k,p}(\omega_{k})}^q
d\varepsilon\le e^{-(k+1)}(k+1)^{-N_k}.\]
Hence, for each $s>0$,
$$
\int_{\frac{1}{k+1}}^{\frac{1}{k}}\varepsilon^{-s}\Vert f_\varepsilon-g_{\varepsilon,k}\Vert_{W^{k,p}(\omega_k)}^qd\varepsilon
\le
 (k+1)^{s}e^{-(k+1)}.
$$
The net $(g_{\varepsilon})_{\varepsilon}$ defined as $g_{\varepsilon}=g_{\varepsilon,k}$ if $\varepsilon\in [1/(k+1), 1/k]$ satisfies all requirements.

\end{proof}

\section{Local Besov type subspaces of $\mathcal{G}_{q}(\Omega)$}
\label{section l besov gf}

We define in this short section subspaces of $\mathcal{G}_{q}(\Omega)$ that  will be shown to correspond to $B_{p,q,loc}^{s}(\Omega)$ and $C^{\infty}(\Omega)$ in a precise fashion. We are interested in nets $(f_\varepsilon)_\varepsilon\in{\mathcal E}_{q}(\Omega)
$ 
such that for given $k\in \mathbb {N}$ and $s\in\mathbb{R}$
\begin{equation}
\label{eqnetg1}
(\forall \omega\Subset \Omega)\Big(\int_0^1\varepsilon^{sq}||f_\varepsilon||_{W^{k,p}(\omega)}^q\frac{d\varepsilon}{\varepsilon}<\infty\Big).
\end{equation}
(With the obvious change when $q=\infty$, namely,
\begin{equation}\label{eqnetg11}
(\forall \omega\Subset \Omega)(\sup_{\varepsilon\in (0,1)}
\varepsilon^{s}||f_\varepsilon||_{W^{k,p}(\omega)}<\infty).)
\end{equation}

\begin{definition} \label{def1} Let $s\in\mathbb{R}$ and $k\in\mathbb{N}_{0}\cup\left\{\infty\right\}$. 
\begin{itemize}
\item[(i)] A  net $(f_{\varepsilon})_{\varepsilon}\in \mathcal{E}_{q}(\Omega)$ is said to belong to $\mathcal{E}^{k,-s}_{q,p}(\Omega)$ if \eqref{eqnetg1} holds (correspondingly \eqref{eqnetg11} if $q=\infty$).
\item[(ii)] $\mathcal{G}^{k,-s}_{q,p}(\Omega)=\{f\in \mathcal{G}_{q}(\Omega): (\exists (f_{\varepsilon})_{\varepsilon}\in \mathcal{E}^{k,-s}_{p,q}(\Omega))\:(f=[(f_{\varepsilon})_{\varepsilon}]) \}.$
\item [(iii)] $\mathcal{G}^{\infty}_{q}(\Omega)=\{f\in \mathcal{G}_{q}(\Omega): (\forall \omega\Subset \Omega)(\exists s\in\mathbb{R})(\forall k\in\mathbb{N})(f_{|\omega}\in \mathcal{G}^{k,-s}_{q,p}(\omega)) \}.$
\end{itemize}
\end{definition}

\medskip

Let us conclude this section with some remarks:

\begin{remark}
It is not hard to see that the following properties hold, we leave their verifications as an exercise to the reader.
\begin{itemize}
\item [(i)] $
{\mathcal{G}}_{q,p}^{k,-s}(\Omega)\subseteq
{\mathcal{G}}_{q,p}^{k_1,-s_1}(\Omega)$ if and only if $k\geq k_1$ and $s\leq
s_1$.
\item [(ii)] Let $P(D)$ be a
differential operator of order $m\leq k$ with constant coefficients. Then
$P(D):{\mathcal{G}}_{q,p}^{k,-s}(\Omega)\rightarrow
{\mathcal{G}}_{q,p}^{k-m,-s}(\Omega).$
\item [(iii)] The definition of $\mathcal{G}^{\infty}_{q}$ is independent of the choice of $p\in[1,\infty]$ (Sobolev's embedding lemma).
\item [(iv)] $\mathcal{G}^{\infty}_{q}$ is a subsheaf of $\widetilde{\mathbb{C}}_{\infty}$-submodules of $\mathcal{G}_{q}$.
\item [(v)] When $q=\infty$, if we regard $\mathcal{G}_{\infty}(\Omega)=\mathcal{G}_{co}(\Omega)$ as a subalgebra of the Colombeau algebra $\mathcal{G}(\Omega)$, we have that $\mathcal{G}_{\infty}^{\infty}(\Omega)=\mathcal{G}^{\infty}(\Omega)\cap \mathcal{G}_{\infty}(\Omega)$, where $\mathcal{G}^{\infty}(\Omega)$ is Oberguggenberger's algebra \cite{ober001} of regular generalized functions.
\end{itemize}
\end{remark}

\section{Characterization of Besov regularity for distributions}
\label{cass}
The goal of this section is to provide a characterization of those generalized functions in $\mathcal{G}_{q}(\Omega)$ that arise from elements of the local Besov spaces of distributions. Our characterization is in terms of the spaces $\mathcal{G}^{k,-s}_{q,p}(\Omega)$ introduced in Section \ref{section l besov gf}. We mention that Theorem \ref{1mtheorem} extends our characterization of Zygmund spaces from \cite{psv}.

\begin{theorem}
\label{1mtheorem} Let $s>0$. We have $\mathcal{G}^{k,-s}_{q,p}(\Omega)\cap\iota(\mathcal{D}'(\Omega)) =\iota(B_{p,q,loc}^{k-s}(\Omega))$. 
\end{theorem}

Naturally, Theorem \ref{1mtheorem} might be rephrased as follows:

\begin{corollary}
\label{mcor1} Let $r\in\mathbb{R}$. If $k$ is any non-negative integer such that $k>r$, then $$\iota(B_{p,q,loc}^{r}(\Omega))=\mathcal{G}^{k,r-k}_{q,p}(\Omega)\cap\iota(\mathcal{D}'(\Omega)).$$
\end{corollary}

\begin
{proof}[Proof of Theorem \ref{1mtheorem}]
We first point out that the statement of Theorem \ref{1mtheorem} is a local one. Thus, it is enough to show that 
$\iota(B_{p,q}^{k-s}(\mathbb{R}^{d})\cap\mathcal{E}'(\Omega))=\mathcal{G}^{k,-s}_{q,p}(\Omega)\cap\iota(\mathcal{E}'(\Omega)) $.

Assume $T\in B_{p,q}^{k-s}(\mathbb{R}^{d})\cap \mathcal{E}'(\Omega)$. Since partial derivatives continuously act on Besov spaces \cite{peetre76}, we obtain $T^{(\alpha)}\in B_{p,q}^{k-|\alpha|-s}(\mathbb{R})\subseteq B_{p,q}^{-s}(\mathbb{R}^{d})$ for all $|\alpha|\leq k$. It remains to notice that $(\phi,\phi)$ is an LP-pair of order $-s$, so that

$$\int_{0}^{1}\varepsilon^{sq} ||T^{(\alpha)}*\phi_\varepsilon||^{q}_{L^p(\mathbb{R}^{d})} \frac{d\varepsilon}{\varepsilon}<\infty,
$$
whence $\iota(T)\in\mathcal{G}^{k,-s}_{q,p}(\Omega)$.

To prove the reverse inclusion, we now suppose that $T\in\mathcal{E}'(\Omega)$ and $((T\ast\phi_{\varepsilon})_{|\Omega})_{\varepsilon}\in \mathcal{E}^{k,-s}_{q,p}(\Omega)$. We consider the net of smooth functions $g_{\varepsilon}=\varepsilon^{s}(T\ast \phi_{\varepsilon})$. Using the fact that $T$ is compactly supported, our hypothesis then reads
\begin{equation}
\label{eq aux 1}
\int_{0}^{1} \|g_{\varepsilon}\|^{q}_{W^{k,p}(\mathbb{R}^{d})}\frac{d\varepsilon}{\varepsilon}<\infty.
\end{equation}
We now use the embedding theorem \cite[Theorem 4, p.~64]{peetre76}, according to which we have that the inclusion mapping $W^{k,p}(\mathbb{R}^{d})\to B^{k}_{p,\infty}(\mathbb{R}^{d})$ is continuous, namely, there is a constant independent of the net such that  $\|g_{\varepsilon}\|_{B^{k}_{p,\infty}(\mathbb{R}^{d})}\leq C\|g_{\varepsilon}\|_{W^{k,p}(\mathbb{R}^{d})}$.  Therefore, by \eqref{eq aux 1}, 
\begin{align*}
&\varepsilon^{s}||T\ast\phi\ast\phi_{\varepsilon}||_{L^{p}(\mathbb{R}^{d})}+\int_{0}^{1}\left(\sup_{y\in(0,1]}\varepsilon^{s}y^{-k}||T\ast\phi_{\varepsilon}\ast\psi_{y}||_{L^{p}(\mathbb{R}^{d})}\right)^{q}\frac{d\varepsilon}{\varepsilon}
\\
&
\qquad= ||g_{\varepsilon}\ast\phi||_{L^{p}(\mathbb{R}^{d})}+\int_{0}^{1}\left(\sup_{y\in(0,1]}y^{-k}||g_{\varepsilon}\ast\psi_{y}||_{L^{p}(\mathbb{R}^{d})}\right)^{q}\frac{d\varepsilon}{\varepsilon}<\infty,
\end{align*}
where we choose $
\psi\in\mathcal{S}(\mathbb{R}^{d})$ such that $(\phi,\psi)$ forms an LP-pair of order $k$ (cf. (\ref{lpcond1}) and (\ref{lpcond2})). Setting $\varepsilon=1$ in the first term of the above expression, $y=\varepsilon$ in the integral to remove the supremum on $y$, $\phi_1=\phi\ast\phi$, and $\psi_1= \phi\ast\psi$, and noticing that $(\phi_1,\psi_{1})$ is again a Littlewood-Paley pair, we obtain
$$||T\ast\phi_1||_{L^{p}(\mathbb{R}^{d})}+\int_{0}^{1}\varepsilon^{q(s-k)}||T\ast(\psi_{1})_{\varepsilon}||_{L^{p}(\mathbb{R}^{d})}\frac{d\varepsilon}{\varepsilon}<\infty,
$$
which yields $T\in B^{k-s}_{p,q}(\mathbb{R}^{d}) $ and completes the proof of the theorem.
\end{proof}

We can use Theorem \ref{mcor1} to give a generalization of Oberguggenberger's regularity result \cite{ober001} for his classical algebra $\mathcal{G}^{\infty}(\Omega)$.

\begin{corollary}
\label{mcor2} We have $\iota(\mathcal{D}'(\Omega))\cap \mathcal{G}^{\infty}_{q}(\Omega)=\iota(C^{\infty}(\Omega))$.
\end{corollary}
\begin{proof} One inclusion is obvious. For the other inclusion, localizing, it suffices to show that if $T\in\mathcal{E}'(\Omega)$ and $\iota(T)\in \mathcal{G}^{\infty,-s}_{q,\infty}(\Omega)$ for some $s>0$, then $T\in C^{\infty}(\Omega)$. Given any $k>0$, Theorem \ref{1mtheorem} yields $T\in B_{\infty,q,loc}^{k-s}(\Omega)$, that is, $T\in\bigcap_{k} B_{\infty,q,loc}^{k-s}(\Omega)=C^{\infty}(\Omega)$. 
\end{proof}

\section{Regularity criteria through association} \label{section association}
In general a net $(f_{\varepsilon})_{\varepsilon}\in\mathcal{E}(\Omega)^{(0,1)}$ is said  \cite{col1,ober001} to be associated to a distribution $T\in \mathcal{D}'(\Omega)$
if 
\begin{equation}
\label{eq asso}
f_{\varepsilon}=T+o(1) \quad \mbox{in }\mathcal{D}'(\Omega),
\end{equation}
that is, if for each test function $\rho\in\mathcal{D}(\Omega)$
\[\lim_{\varepsilon\to0^{+}} \langle f_{\varepsilon}, \rho \rangle =\langle T, \rho \rangle.
\]

Interestingly, if one strengthens the speed of convergence in \eqref{eq asso}, sometimes growth rate properties of the net are reflected in regularity properties of the distributions.  A concept that has proved useful in this respect is that of strong association. In fact, many regularity criteria for distributions in the non-linear theory of generalized functions are based on this concept, see \cite{ps,PSV,piscavi,psv}.

The next definition provides a new concept of association that is applicable to the context of our module $\mathcal{G}_{q}(\Omega)$. The special case $q=\infty$ recovers the classical strong association.

\begin{definition}\label{qaa}
Let $T\in\mathcal D'(\Omega)$ and $f\in \mathcal \mathcal \mathcal{G}_q(\Omega)$. We say that $f$ is strongly $q$-associated to $T$ (on $\Omega$) if there is a representative $(f_{\varepsilon})_{\varepsilon}$, that is, $f=[(f_\varepsilon)_\varepsilon]$, such that
 \begin{equation}\label{ass}
(\forall \omega\Subset \Omega)(\exists b>0)(\forall \rho \in \mathcal{D}(\omega))\Big(\int_0^1\varepsilon^{-bq}|\langle T-f_\varepsilon, \rho
\rangle|^q\frac{d\varepsilon}{\varepsilon}<\infty\Big).
\end{equation}
\end{definition}
Note that the definition does not depend on the representative. In the first part of this section, we derive some regularity results for distributions using properties of strongly $q$-associated generalized functions to it. The second part discusses a characterization of equality in $\mathcal{G}_{q}(\Omega)$ using a faster notion of association. We shall need the following useful lemma.

\begin{lemma}
\label{lemma asso}
Let $T\in\mathcal{E}'(\Omega)$ and 
$f=[(f_\varepsilon)_\varepsilon]\in\mathcal G_q(\Omega)$  be strongly $q$-associated. Then, there are an open set $\omega\Subset\Omega$, $k_{0}\in\mathbb{N}_{0}$, $M>0$, and $b>0$ such that 
\begin{equation}
\label{eqnet222}
(\forall \rho \in  C^\infty(\mathbb{R}^{d}) )\Big(\int_0^1t^{-bq}|\langle T-f_t, \rho
\rangle|^q\frac{dt}{t}\leq \Big( M \sup_{y\in \overline{\omega}, |\alpha|\leq k_0}|\rho^{(\alpha)}(y)|\Big)^{q} \Big).
\end{equation}
\end{lemma}
\begin{proof}
Since $T$ has compact support, we might assume that also the net is supported on a fixed compact set. So, let $\omega\Subset \Omega$
be such that $T$ and each $f_{\varepsilon}$ have supports inside $\omega$.
 Let $M>0$ and consider the sets
$$
X_{M,b}=\Big\{\rho\in C^\infty(\overline{\omega}):
\int_0^1\varepsilon^{-bq}|\langle T-f_\varepsilon, \rho
\rangle|^q\frac{d\varepsilon}{\varepsilon}\leq M\Big\}.
$$
By the Beppo Levi theorem, we have that $X_{M,b}=\bigcap_{n=2}^{\infty} X_{M,b,n},$
where $X_{M,b,n}$ is the closed set
$$X_{M,b,n}=\Big\{\rho\in C^\infty(\overline{\omega}):
\int_{1/n}^1\varepsilon^{-bq}|\langle T-f_\varepsilon, \rho
\rangle|^q\frac{d\varepsilon}{\varepsilon}\leq M\Big\}.$$
Our hypothesis is that $C^\infty(\overline{\omega})=\bigcup_{M,b>0} X_{M,b}$, thus, by the Baire theorem, for some $M,b>0$ one of sets $X_{M,b}$ does not have  empty interior, whence (\ref{eqnet222}) follows.

\end{proof}

 \subsection{Some regularity criteria based on strong $q$-association}
 The following theorem provides two regularity criteria for distributions. We recall that when $p=q=\infty$, the Besov spaces coincide with the H\"{o}lder-Zygmund spaces, that is, 
 $B_{\infty,\infty}^{r}(\mathbb{R}^{d})=C^{r}_{\ast}(\mathbb{R}^{d})$ for all $r\in\mathbb{R}$.
 
 \begin{theorem}\label{not1} Let $T\in\mathcal{D}'(\Omega)$ and 
$f\in\mathcal G_q(\Omega)$  be strongly $q$-associated.
\begin{enumerate}
\item[(i)] If $f\in \mathcal{G}^{\infty}_{q}(\Omega)$, then $T\in C^{\infty}(\Omega)$.
\item [(ii)] If $f\in \mathcal G^{k,-s}_{q, L^{\infty}_{loc}}(\Omega)$ for every $s>0,$ then
  $T\in C^{k-s}_{\ast, loc}(\Omega),$ for every $s>0$.
\end{enumerate}
\begin{remark}
Clearly, by H\"{o}lder's inequality, we have the embeddings $B^{s_1}_{p,q_{1}}(\mathbb{R}^{d})\subset B^{s_2}_{p,q_{2}}(\mathbb{R}^{d})$ whenever $s_{1}> s_{2}$ and $q_1\geq q_2$. Therefore, we also obtain in Theorem \ref{not1}(ii) that $T\in B^{k-s}_{\infty,r, loc}(\Omega),$ for every $s>0$ and $ r\in[1,\infty]$. 
\end{remark}
 
\end{theorem}

Since Theorem \ref{not1} is of local nature, it is a consequence of the next lemma in combination with Lemma \ref{lemma asso}. In fact, part (i) follows by taking $k\to\infty$, while (ii) is deduced by taking $s\to0^{+}$.

\begin{lemma}
Suppose that $T\in\mathcal{E}'(\Omega)$ satisfies \eqref{eqnet222} for given $k_0\in\mathbb{N}_{0}$, $b>0$, $M$, and $\omega\Subset \Omega$. If $(f_{\varepsilon})_{\varepsilon}\in \mathcal E^{k,-s}_{\infty,q}(\Omega)$, then $T\in C_{\ast,loc}^{k-s_0}(\mathbb{R}^{d})$ where $s_0=s(k+d+k_{0})/(s+b)$.
\end{lemma}
\begin{proof} We might assume that all $f_{\varepsilon}$ have supports contained in a fixed compact set. Let $0<\eta\leq 1/2$ and let $q'$ be such that $1/q+1/q'=1.$ Using H\"{o}lder's inequality and \eqref{eqnet222}, we obtain

\begin{align*}
(\log 2 )\|T\ast \phi_{\varepsilon}\|_{W^{k,\infty}(\mathbb{R}^{d})}&\leq \int_{\eta}^{2\eta} \| (T-f_{t})\ast \phi_{\varepsilon} \|_{W^{k,\infty}(\mathbb{R}^{d})}\frac{dt}{t}+\int_{\eta}^{2\eta} \|f_{t}\ast \phi_{\varepsilon} \|_{W^{k,\infty}(\mathbb{R}^{d})}\frac{dt}{t}
\\
&
\leq \left(\int_{\eta}^{2\eta} t^{bq'-1}dt\right)^{1/q'} \left(\int_{0}^{1} t^{-bq} \| (T-f_{t})\ast \phi_{\varepsilon} \|^{q}_{W^{k,\infty}(\mathbb{R}^{d})}\frac{dt}{t}\right)^{1/q}
\\
&
\qquad \qquad \quad  +\|\phi\|_{L^{1}(\mathbb{R})}\int_{\eta}^{2\eta} \|f_{t} \|_{W^{k,\infty}(\mathbb{R}^{d})}\frac{dt}{t}
\\
&\leq 
(2\eta)^{b}(\log 2)^{1/q'} M \varepsilon^{-d-k-k_{0}} \|\phi\|_{W^{k+k_0,\infty}(\mathbb{R}^{d})}
\\
&
\qquad \qquad \quad +\eta^{-s}(\log 2)^{1/q'}\|\phi\|_{L^{1}(\mathbb{R})}\left(\int_{0}^{1} t^{sq}\|f_{t}\|^{q}_{W^{k,\infty}(\mathbb{R}^{d})}\frac{dt}{t}\right)^{1/q}.
\end{align*}
Hence, 
$$
\|T\ast \phi_{\varepsilon}\|_{W^{k,\infty}(\mathbb{R}^{d})}\leq C(\eta^b\varepsilon^{-d-k-k_{0}}+\eta^{-s}),
$$
where $C>0$. In order to minimize this expression, we choose $\eta=\varepsilon^{(d+k+k_{0})/(s+b)}$ so that
\[
\|T\ast \phi_{\varepsilon}\|_{W^{k,\infty}(\mathbb{R}^{d})}\leq 2C\varepsilon^{-s(d+k+k_{0})/(s+b)}.
\]
The result then follows from Theorem \ref{1mtheorem} in view of this inequality.
\end{proof}

\subsection{A characterization of equality in $\mathcal{G}_{q}(\Omega)$} In this subsection
we study  a strengthened version of the strong $q$-association. We denote as $\mathfrak{S}$ the class of rapidly decreasing functions at $0^{+}$, that is, a measurable function $R:(0,1)\to (0,\infty)$ is said to belong to $\mathfrak{S}$ if for each $b$ we have $R(\varepsilon)=O(\varepsilon^{b})$ as $\varepsilon\to0^{+}$.
\begin{definition} Two generalized functions $f=[(f_\varepsilon)_{\varepsilon}]\in\mathcal{G}_{q}(\Omega)$ and $g=[(g_\varepsilon)_{\varepsilon}]\in\mathcal{G}_{q}(\Omega)$ are said to be rapidly $q$-associated if  \begin{equation}\label{ras}
(\forall \omega\Subset \Omega)(\exists R \in \mathfrak{S})(\forall \rho \in \mathcal{D}(\omega))\Big(\int_0^1\frac{|\langle f_\varepsilon-g_\varepsilon, \rho
\rangle|^q}{R(\varepsilon)}d\varepsilon<\infty\Big).
\end{equation}

\end{definition}

The next proposition generalizes \cite[Theorem 4]{PSV}.

\begin{proposition}\label{rd}
Let $(f_\varepsilon)_{\varepsilon}, (g_\varepsilon)_{\varepsilon}\in\mathcal E_{q}(\Omega)$ be rapidly $q$-associated. 
Then, $[(f_\varepsilon)_\varepsilon]= [(g_\varepsilon)_{_\varepsilon}]$ in $\mathcal G_q(\mathbb R^d).$
\end{proposition}
\begin{proof}
Subtracting one from another, we might simply assume that $g_{\varepsilon}=0$ for all $\varepsilon$. We then have to show that $(f_\varepsilon)_\varepsilon\in \mathcal{N}_{q}(\Omega)$. Since the statement and the hypothesis are local, we might actually assume that all $f_{\varepsilon}$ have supports contained in a fixed relatively open set $\omega\Subset \Omega$. 
Similarly as in Lemma \ref{lemma asso}, 
 there exist $k_0\in \mathbb N$, $M>0$, and $R\in\mathfrak{S}$ such that
\begin{equation}
\label{eqnet333}
(\forall \rho\in C^{k_{0}}(\mathbb R^d))\Big(
\int_0^1\frac{|\langle f_\varepsilon, \rho
\rangle|^q}{R(\varepsilon)}d\varepsilon \leq  M\sup_{y\in \omega, |\alpha|\leq k_0}|\rho^{(\alpha)}(y)|\Big).
\end{equation}
The next step is to prove that  $f_\varepsilon\in\mathcal{N}_{q}(\Omega)$ and for this  we use an idea from the proof of \cite[Theorem 4]{PSV}. In fact, we employ the powerful Schwartz parametrix method \cite{schw}. There are $m\in \mathbb{N},$ $\chi\in \mathcal{D}^{k_{0}}(\mathbb{R}^d)$, and $\theta \in \mathcal{D}(\mathbb{R}^d)$ such that $\delta=\Delta^m\chi+ \theta$, $\operatorname*{supp }\chi + \omega\Subset \Omega$, and $\operatorname*{supp }\theta + \omega\Subset \Omega$. We then obtain the representation
\begin{equation}
\label{par} 
 f_{\varepsilon} =\Delta^m (f_{\varepsilon}*\chi)+ f_{\varepsilon}*\theta, \qquad \varepsilon\in(0,1).
\end{equation}
Applying \eqref{eqnet333} and using that $R\in\mathfrak{S}$, we now obtain that for each $b>0$

$$\int_0^1\varepsilon^{-b}||f_\varepsilon*\chi||^q_{L^\infty(\mathbb R^d)}\frac{d\varepsilon}{\varepsilon} + \int_0^1\varepsilon^{-b}||f_\varepsilon*\theta||^q_{L^\infty(\mathbb R^d)}\frac{d\varepsilon}{\varepsilon}<\infty.
$$
Lemma \ref{lemma null ideal q} now implies that $(f_{\varepsilon}\ast \chi)_{\varepsilon}, (f_{\varepsilon}\ast \theta)_{\varepsilon}\in \mathcal{N}_{q}(\Omega)$. The formula \eqref{par} hence yields $(f_{\varepsilon})_{\varepsilon}\in \mathcal{N}_{q}(\Omega)$, as claimed.
\end{proof}

\section{Global type spaces of generalized functions}\label{section global}
We now define global counterparts on $\mathbb{R}^{d}$ of the local spaces we have been considering in the previous sections. Let $k\in\mathbb{N}_{0}$ and $s\in\mathbb{R}$. The basic blocks are the global spaces of nets (with the obvious change when $q=\infty$)
$$\mathcal{E}^{k,-s}_{q,L^p}(\mathbb{R}^{d})=\left\{f\in \mathcal{E}_{q}(\mathbb{R}^d): \: 
\int_0^1\varepsilon^{qs}||f_\varepsilon||_{W^{k,p}(\mathbb{R}^{d})}^q\frac{d\varepsilon}{\varepsilon}<\infty
 \right\}.$$
 We then set 
 $$\mathcal{E}_{q,L^p}(\mathbb{R}^{d})=\bigcap_{k}\bigcup_{s}\mathcal{E}^{k,-s}_{q,L^p}(\mathbb{R}^{d}), \quad \mathcal{E}^{\infty}_{q,L^p}(\mathbb{R}^{d})=\bigcup_{s}\bigcap_{k} \mathcal{E}^{k,-s}_{q,L^p}(\mathbb{R}^{d}), $$
 and
$$ \mathcal{N}_{q,L^p}(\mathbb{R}^{d})=\bigcap_{k,s}\mathcal{E}^{k,-s}_{q,L^p}(\mathbb{R}^{d}),$$
 and definite our global $\widetilde{\mathbb{C}}_{\infty}$-modules of generalized functions as the quotients
 $$
 \mathcal{G}_{q,L^p}(\mathbb{R}^{d})= \mathcal{E}_{q,L^p}(\mathbb{R}^{d})/\mathcal{N}_{q,L^p}(\mathbb{R}^{d}) \quad \supset \quad  \mathcal{G}^{\infty}_{q,L^p}(\mathbb{R}^{d})= \mathcal{E}^{\infty}_{q,L^p}(\mathbb{R}^{d})/\mathcal{N}_{q,L^p}(\mathbb{R}^{d}).
 $$
 
We note that $\mathcal{G}_{\infty,L^{\infty}}(\mathbb{R}^{d})$ and $\mathcal{G}^{\infty}_{\infty,L^{\infty}}(\mathbb{R}^{d})$ are differentiable algebras. Moreover,  $\mathcal{G}_{q,L^p}(\mathbb{R}^{d})$ is a differential module over $\mathcal{G}_{\infty,L^{\infty}}(\mathbb{R}^{d})$, while $\mathcal{G}^{\infty}_{q,L^p}(\mathbb{R}^{d})$ is a differential module over $\mathcal{G}^{\infty}_{\infty,L^{\infty}}(\mathbb{R}^{d})$. We also mention that $\mathcal{G}_{\infty,L^{p}}(\mathbb{R}^{d})$ coincides with the global algebra denoted as $\mathcal{G}_{p,p}(\mathbb{R}^{d})$ by Biagioni and Oberguggenberger \cite{b-o}.

Certain important distributions can be embedded into our global modules of generalized functions. The natural classes to consider here are the Schwartz $L^p$-based distribution spaces \cite{schw} (cf. \cite{dpv15}):
 
$$
\mathcal{D}'_{L^{p}}(\mathbb{R}^{d})=\bigcup_{k\in\mathbb{N}}W^{-k,p}_{0}(\mathbb{R}^{d}) \quad \mbox{ and }\quad \mathcal{D}_{L^{p}}(\mathbb{R}^{d})=\bigcap_{k\in\mathbb{N}}W^{k,p}(\mathbb{R}^{d}). 
$$
It is not difficult to show that $\mathcal{D}'_{L^{p}}(\mathbb{R}^{d})$ embeds into $\mathcal{G}_{q,L^{p}}(\mathbb{R}^{d})$ via mollification (cf. Subsection \ref{notation})
$$\iota:\mathcal{D}'_{L^{p}}(\mathbb{R}^{d})\mapsto \mathcal{G}_{q,L^{p}}(\mathbb{R}^{d}), \quad \iota(T)=[(T\ast\phi_{\varepsilon})_{\varepsilon}].$$ The next proposition actually shows that $\mathcal{D}'_{L^{p}}(\mathbb{R}^{d})$ is the largest subspace of $\mathcal{S}'(\mathbb{R}^{d})$ that can be embedded into  $\mathcal{G}_{q,L^{p}}(\mathbb{R}^{d})$ in this fashion.

\begin{proposition}
\label{theoremLp} Let $T\in\mathcal{S}'(\mathbb{R}^{d})$. If $(T\ast\phi_{\varepsilon})_{\varepsilon}\in \mathcal{E}_{q,L^{p}}(\mathbb{R}^{d})$, then $T\in\mathcal{D}'_{L^{p}}(\mathbb{R}^{d})$.
\end{proposition}   
\begin{proof} Using the convolution average characterization
 of $\mathcal{D}'_{L^{p}}(\mathbb{R}^{d})$ (cf. \cite{dpv15,schw}), it suffices to prove that $T\ast \rho \in L^p(\mathbb{R}^{d})$ for each $\rho\in\mathcal{S}(\mathbb{R}^{d})$. We show this via the vector-valued Tauberian theory for class estimates from \cite{pv2019}. Define the vector-valued tempered distribution $\mathbf{T}$ as $\left\langle \mathbf{T},\rho\right\rangle:=T\ast \rho$ for test functions $\rho\in\mathcal{S}(\mathbb{R}^{d})$. We have that $\mathbf{T}\in\mathcal{S}'(\mathbb{R}^{d},\mathcal{S}'(\mathbb{R}^{d}))$ and we must show that $\mathbf{T}\in\mathcal{S}'(\mathbb{R}^{d},L^p(\mathbb{R}^{d}))$. Being a convolution operator, $T$ intertwines translations. Since $(\phi,\phi)$ is an LP-pair of order $-s$ for any $s> 0$, our hypothesis $(T\ast\phi_{\varepsilon})_{\varepsilon}\in \mathcal{E}_{q,L^{p}}(\mathbb{R}^{d})$ is a particular case of the assumptions from \cite[Theorem 6.1]{pv2019}, so that the latter directly yields the desired membership $\mathbf{T}\in\mathcal{S}'(\mathbb{R}^{d},L^p(\mathbb{R}^{d}))$.
\end{proof}

We end this article by giving global versions of the results from Section \ref{cass}, where in particular we characterize the global Besov spaces.

\begin{theorem}
\label{theoremextra} Let $r\in\mathbb{R}$ and $s>0$.
\begin{itemize}
\item [(i)] We have $\mathcal{G}^{k,-s}_{q,L^{p}}(\mathbb{R}^{d})\cap\iota(\mathcal{D}'_{L^{p}}(\mathbb{R}^{d})) =\iota(B^{k-s}_{p,q}(\mathbb{R}^{d}))$. 
\item [(ii)] For any  integer $k>r$, we have $\iota(B_{p,q}^{r}(\mathbb{R}^{d}))=\mathcal{G}^{k,r-k}_{q,L^{p}}(\mathbb{R}^{d})\cap \iota(\mathcal{D}'_{L^{p}}(\mathbb{R}^{d}))$.
\item [(iii)] We have  $\iota(\mathcal{D}_{L^p}(\mathbb{R}^{d}))=\mathcal{G}^{\infty}_{q,L^p}(\mathbb{R}^{d})\cap\iota(\mathcal{D}'_{L^{p}}(\mathbb{R}^{d})) .$
\end{itemize}
\end{theorem}
\begin{proof}
The property (ii) is a reformulation of (i). Since $\mathcal{D}_{L^p}(\mathbb{R}^d)=\bigcap_{r}B^{r}_{p,q}(\mathbb{R}^{d})$, we obtain that (iii) follows at once from (i). The proof of (i) is a straightforward modification of the proof of Theorem \ref{1mtheorem}, which we therefore omit.
\end{proof}

\bigskip

\subsection*{Acknowledgements} We thank Hans Vernaeve for useful discussions on the subject. 

\bigskip



\begin{thebibliography}{2000}

\bibitem{adams} R.~A.~Adams, J.~Fournier, {\em Sobolev spaces}, second edition, Pure and Applied Mathematics, 140, Elsevier/Academic Press, Amsterdam, 2003.

\bibitem{bia} H.~A.~Biagioni, {\em A nonlinear theory of generalized
functions}, Springer, Berlin-Hedelberg-New York, 1990.


\bibitem{b-o} H.~A.~Biagioni, M.~Oberguggenberger, {\em Generalized solutions to the Korteweg-de Vries and the regularized long-wave equations,} SIAM J. Math. Anal. 23 (1992), 923--940.

\bibitem{bucu} A.~Burtscher, M.~Kunzinger, Algebras of generalized functions with smooth parameter dependence,  Proc. Edinb. Math. Soc. 55 (2012),  105--124.

\bibitem{col1}
J.-F. Colombeau,
{\em Elementary introduction to new generalized functions},
North-Holland Math. Stud. 113, North-Holland Publishing Co., Amsterdam, 1985.

\bibitem{dvv18}  A.~Debrouwere, H.~Vernaeve, J.~Vindas, Optimal embeddings of ultradistributions into differential algebras, Monatsh. Math. 186 (2018), 407--438. 

\bibitem{dvv19} A.~Debrouwere, H.~Vernaeve, J.~Vindas, A nonlinear theory of infrahyperfunctions, Kyoto J. Math. 59 (2019), 86--895.

\bibitem{dpv15}  P.~Dimovski, S.~Pilipovi\'{c}, J.~Vindas, New distribution spaces associated to translation-invariant Banach spaces, Monatsh. Math. 177 (2015), 495--515.

\bibitem{gah} C. Garetto, G. H\" ormann, Microlocal analysis
of generalized functions: pseudodifferential techniques and propagation of singularities,
Proc. Edinb. Math. Soc. 48 (2005), 603--629.


\bibitem{gkos}M. Grosser, M. Kunzinger, M. Oberguggenberger, R. Steinbauer,
{\em Geometric theory of generalized functions with applications to general relativity}, Mathematics and its Applications, 537, Kluwer Acad. Publ., Dordrecht, 2001.


\bibitem{komatsu73} H.~Komatsu, Relative cohomology of sheaves of solutions of differential equations, in: \emph{Hyperfunctions and pseudo-differential equations,} pp. 192---261, Lecture Notes in Math. 287, Springer, Berlin, 1973.

\bibitem{hoermann} G. H\"{o}rmann, H\"{o}lder-Zygmund regularity in algebras
of generalized functions, Z. Anal. Anwend. 23 (2004), 139--165.

\bibitem{kuhe}  G. H\"{o}rmann, M. Kunzinger, Microlocal properties of basic operations in Colombeau algebras, J. Math. Anal. Appl. 261 (2001), 254--270.


\bibitem{nps}  M. Nedeljkov,  S. Pilipovi\'{c}, D.~Scarpal\'{e}zos, {\em The linear theory of Colombeau generalized functions}, Pitman Res. Notes Math. Ser. 385, Longman, Harlow, 1998.

\bibitem{ober001}
M.\ Oberguggenberger, {\em Multiplication of distributions and
applications to partial differential equations}, Pitman Res. Notes
Math. Ser. 259, Longman, Harlow, 1992.


\bibitem{peetre76} J.~Peetre, \emph{New thoughts on Besov spaces,} Duke University Mathematics Series, No. 1, Duke University, Mathematics Department, Durham, N.C., 1976.


\bibitem{ps}S.~Pilipovi\'{c}, D.~Scarpal\'{e}zos, Regularity properties of distributions and ultradistributions, Proc. Amer. Math. Soc. 129 (2001), 3531--3537.

\bibitem{PSV}  S. Pilipovi\'{c}, D. Scarpal\'{e}zos, V. Valmorin, Equalities in algebras of generalized functions, Forum Math. 18 (2006), 789--802.

\bibitem{piscavi}  S. Pilipovi\'{c}, D. Scarpal\'{e}zos, J. Vindas, Regularity properties of distributions through sequences of functions, Monatsh. Math. 170 (2013), 307--322.

\bibitem{psv}  S. Pilipovi\'{c}, D. Scarpal\'{e}zos, J. Vindas,   Classes of generalized functions with finite type regularities, in:
{\em Pseudo-differential operators, generalized functions and asymptotics,} 307--322, Oper. Theory Adv. Appl., 231, 
Birkh\" auser Springer Basel AG, Basel, 2013. 

\bibitem{pv2019} S. Pilipovi\' c, J. Vindas, Tauberian class estimates for vector-valued distributions, Sb. Math. 210 (2019), 272--296.

\bibitem{schw} L. Schwartz, {\em Th\'{e}orie des distributions}, Hermann, Paris, 1966.


\end{thebibliography}
\end{document}